\documentclass{amsart}
\usepackage{latexsym,amsxtra,amscd,ifthen}
\usepackage{amsfonts}
\usepackage{verbatim}
\usepackage{amsmath}
\usepackage{amsthm}
\usepackage{amssymb}

\numberwithin{equation}{section}

\theoremstyle{plain}
\newtheorem{theorem}{Theorem}[section]
\newtheorem{lemma}[theorem]{Lemma}
\newtheorem{proposition}[theorem]{Proposition}
\newtheorem{corollary}[theorem]{Corollary}

\theoremstyle{definition}
\newtheorem{definition}[theorem]{Definition}
\newtheorem{example}[theorem]{Example}

\newtheorem{remark}[theorem]{Remark}
\newtheorem*{remark*}{Remark}
\newtheorem{question}[theorem]{Question}

\makeatletter              
\let\c@equation\c@theorem  
\makeatother

\DeclareMathOperator{\GKdim}{GKdim}

\DeclareMathOperator{\End}{End}

\DeclareMathOperator{\gr}{gr}

\DeclareMathOperator{\LDA}{{\sf CoLieAlg}}
\DeclareMathOperator{\LA}{{\sf LieAlg}}
\DeclareMathOperator{\CA}{{\sf CoAlg}}
\DeclareMathOperator{\uCA}{{\sf CouAlg}}

\begin{document}

\title{Coassociative Lie algebras}

\author{D.-G. Wang, J.J. Zhang and G. Zhuang}

\address{Wang: School of Mathematical Sciences,
Qufu Normal University, Qufu, Shandong 273165, P.R.China}

\email{dgwang@mail.qfnu.edu.cn, dingguo95@126.com}

\address{Zhang: Department of Mathematics, Box 354350,
University of Washington, Seattle, Washington 98195, USA}

\email{zhang@math.washington.edu}

\address{Zhuang: Department of Mathematics, Box 354350,
University of Washington, Seattle, Washington 98195, USA}

\email{gzhuang@math.washington.edu}

\begin{abstract}
A coassociative Lie algebra is a Lie algebra equipped with a
coassociative coalgebra structure satisfying a compatibility
condition. The enveloping algebra of a coassociative Lie algebra can
be viewed as a coalgebraic deformation of the usual universal
enveloping algebra of a Lie algebra. This new enveloping algebra
provides interesting examples of noncommutative and noncocommutative
Hopf algebras and leads to a classification of connected Hopf algebras
of Gelfand-Kirillov dimension four in \cite{WZZ}.
\end{abstract}





\maketitle

\tableofcontents

{\small Dedicated to Kenny Brown and Toby Stafford on the
occasion of their 60th birthdays.}

\setcounter{section}{-1}
\section{Introduction}
\label{xxsec0}

We introduce the notion of a coassociative Lie algebra which generalizes
in an obvious way both a Lie algebra and a coassociative coalgebra without
counit. The enveloping algebra of a coassociative Lie algebra is a
bialgebra that is a generalization of the usual
universal enveloping algebra of a Lie algebra. The enveloping algebra of
a coassociative Lie algebra should be considered as a coalgebraic
deformation of the usual universal enveloping algebra on one hand, and
potentially as an algebraic deformation of the coordinate ring (regular
functions) of certain algebraic groups or semigroups on the other.

Let $k$ be a base field that is algebraically closed and everything is
over $k$.

Let ${\mathfrak g}$ denote an ordinary Lie algebra and $L$ a
coassociative Lie algebra. Let $U({\mathfrak g})$ (respectively, $U(L)$)
denote the enveloping algebra of ${\mathfrak g}$ (respectively, $L$).
It is well-known that $U({\mathfrak g})$ is a Hopf algebra.
In contrast, $U(L)$ is not a Hopf algebra in general.

\begin{theorem}[Theorem \ref{xxthm2.5}]
\label{xxthm0.1}
The enveloping algebra $U(L)$ of a coassociative Lie algebra $L$ is a Hopf
algebra if and only if $L$ is locally conilpotent.
\end{theorem}

Most of unexplained terms will be defined in Sections 1 and 2.

Starting with conilpotent coalgebras or nilpotnent Lie algebras,
one can construct families of nontrivial coassociative Lie algebras
based on them. For instance, we give explicit examples in Example
\ref{xxex4.4} (based on strictly upper triangular matrix coalgebras)
and Example \ref{xxex4.3} (based on Heisenberg Lie algebras). However, it is
quite unsatisfactory that there is no nontrivial Lie structure based
on cosemisimple coalgebras [Proposition \ref{xxprop3.6}] and no
nontrivial coalgebra structure based on the semisimple Lie algebra
$sl_2$ [Theorem \ref{xxthm3.9}]. It would be interesting if a
generalized version of coassociative Lie structure can be
constructed on cosemisimple coalgebras and/or semisimple Lie
algebras.

Coassociative Lie algebras are helpful in the classification of connected Hopf
algebras of low Gelfand-Kirillov dimension (denoted by $\GKdim$).
Definitions and basic properties of GK-dimension can be
found in the first three chapters of \cite{KL}. Our reference book for
Hopf algebras is \cite{Mo}. Let
$P(H)$ be the subspace of all primitive elements in $H$, and $p(H)$ denote
the dimension of $P(H)$. If $\GKdim H\leq 2$, then it is well-known that
$H\cong U({\mathfrak g})$ for a Lie algebra ${\mathfrak g}$ of dimension
$p(H)$. If $\GKdim H=3$, then it follows from a result of the
third-named author \cite{Zh2} that $H$ is isomorphic to either
\begin{enumerate}
\item[(i)]
$U({\mathfrak g})$ for a Lie algebra ${\mathfrak g}$ of dimension $3$, or
\item[(ii)]
$U(L)$ for a coassociative Lie algebra $L$ of dimension $3$.
\end{enumerate}
In \cite{WZZ}, we give a classification in GK-dimension 4.

\begin{theorem}\cite[Theorem 0.3 and Remark 0.4]{WZZ}
\label{xxthm0.2}
Suppose that $k$ is of characteristic zero.
Let $H$ be a connected Hopf algebra of GK-dimension four. Then one of
the following occurs.
\begin{enumerate}
\item
If $p(H)=4$, then $H\cong U({\mathfrak g})$
for a Lie algebra ${\mathfrak g}$ of dimension $4$.
\item
If $p(H)=3$, then $H\cong U(L)$ for an anti-cocommutative
coassociative Lie algebra $L$ of dimension $4$.
\item
If $p(H)=2$, then $H$ is not isomorphic to either $U({\mathfrak g})$
or $U(L)$ as Hopf algebras, for any Lie algebra ${\mathfrak g}$ or
any coassociative Lie algebra $L$. Such a Hopf algebra is isomorphic
to one of the four families of Hopf algebras explicitly constructed
in \cite[Section 4]{WZZ}.
\end{enumerate}
\end{theorem}

The Hopf algebras in Theorem \ref{xxthm0.2}(b,c) are completely
classified in \cite{WZZ}. The proof of Theorem \ref{xxthm0.2} is
heavily dependent on the study of coassociative Lie algebras. A
result of Milnor-Moore-Cartier-Kostant \cite[Theorem 5.6.5]{Mo}
states that any cocommutative connected Hopf algebra over a field of
characteristic zero is isomorphic to $U({\mathfrak g})$ for some Lie
algebra ${\mathfrak g}$. This applies to case (a) of Theorem
\ref{xxthm0.2}. However, the Hopf algebras in Theorem
\ref{xxthm0.2}(b,c) are not cocommutative, hence not isomorphic to
$U({\mathfrak g})$ for a usual Lie algebra ${\mathfrak g}$. A
generalization of Theorem \ref{xxthm0.2}(b) states that if $p(H)=
\GKdim H-1$, then $H$ is isomorphic to $U(L)$ for some coassociative
Lie algebra $L$ \cite[Theorem 0.5]{WZZ}.

The Hopf algeba $U(\mathfrak g)$ is always involutory. In general
$U(L)$ is not involutory. By using the Calabi-Yau property of the
enveloping algebra of unimodular Lie algebras, we prove the
following result.

\begin{theorem}[Theorem \ref{xxthm2.7}]
\label{xxthm0.3} Let ${\mathfrak g}$ be a finite dimensional
unimodular Lie algebra. Suppose that $L:=({\mathfrak g},\delta)$ is
a conilpotent coassociative Lie algebra. Then $U(L)$ is involutory.
\end{theorem}

Coassociative Lie algebras are not understood fully, and a lot of
basic questions are unsolved. For example,

\begin{question}
\label{xxque0.4} Let $n(L)$ be the nilpotency of $L$ and $con(L)$ be
the conilpotency of $L$. If $n(L)+con(L)<\infty$, then is there the
bound for the number $n(L)+con(L)-\dim L$?
\end{question}

Preliminary analysis of known examples shows that if
$n(L)+con(L)<\infty$, then
$$n(L)+con(L)-\dim L\leq 1.$$

As one may guess, this paper grows out of the study of connected Hopf
algebras \cite{WZZ}. By Theorem \ref{xxthm0.2}(b), certain classes of
connected Hopf algebras are the enveloping algebras of coassociative Lie
algebras, which are noncommutative and noncocommutative. This
construction is very different from the various classical constructions
of quantum groups which we are familiar with. One way of thinking of a
coassociative Lie algebra is that it has an additional
coalgebra structure on a Lie algebra. Our work \cite{WZZ} suggests
that certain ``nilpotent'' quantum groups could be represented by a Lie
algebra with an extra coalgebra structure. This idea might be worth
pursuing further.

\section{Definitions}
\label{xxsec1}
In this section, we start with some definitions and give some easy
examples of coassociative Lie algebras. We would like to thank Milen
Yakimov for suggesting the name ``coassociative Lie algebra''.

\begin{definition}
\label{xxdef1.1} A Lie algebra $(L, [\;,\;])$ together with a coproduct
$\delta: L\to L\otimes L$ is called a {\it coassociative Lie algebra} if
\begin{enumerate}
\item
$(L,\delta)$ is a coassociative coalgebra without counit, and
\item
two operations $\delta$ and $[\;,\;]$ satisfy the following compatibility
condition
\begin{equation}
\label{E1.1.1}\tag{E1.1.1}
\delta([a,b])= b_1\otimes [a, b_2]+[a,b_1]\otimes b_2+
[a_1,b]\otimes a_2+a_1\otimes [a_2,b]+[\delta(a), \delta(b)]
\end{equation}
for all $a,b\in L$.

Here $\delta(x)=x_1\otimes x_2$, which is basically the Sweedler's
notation with the summation indicator omitted.
\end{enumerate}
\end{definition}

Definition \ref{xxdef1.1} is not complete without the following remark.

\begin{remark}
\label{xxrem1.2} For a general Lie algebra $L$, the product
$L\otimes L$ is not a Lie algebra in any natural way. To make sense
of $[\delta(a),\delta(b)]$ in \eqref{E1.1.1}, we need to embed $L$
into the usual universal enveloping algebra $U(L)$. Then $L\otimes
L$ can be naturally identified with a subspace of $U(L)\otimes
U(L)$. Under this identification, $[f,g]$ is defined as the
commutator of $f$ and $g$ in the associative algebra $U(L)\otimes
U(L)$ for all elements $f,g\in L \otimes L$. The equation
\eqref{E1.1.1} implies that the element
$[\delta(a),\delta(b)]\in U(L)\otimes U(L)$ is actually in the
subspace $L\otimes L$, namely,
\begin{equation}
\label{E1.2.1}\tag{E1.2.1}
[\delta(a),\delta(b)]\in L\otimes L.
\end{equation}
\end{remark}

\medskip

Inside $U(L)\otimes U(L)$, the equation \eqref{E1.1.1} can be written as
\begin{equation}
\label{E1.2.2}\tag{E1.2.2}
\delta([a,b])=[a\otimes 1+1\otimes a, \delta(b)]+
[\delta(a),b\otimes 1+1\otimes b]+[\delta(a),\delta(b)]
\end{equation}
for all $a,b\in L$. Let $L^1$ denote the Lie algebra extension $L\oplus k1$
where $1$ is in the center (i.e., $[1, L^1]=0$). Define $\Delta:
L^1\to L^1\otimes L^1$ by
\begin{equation}
\label{E1.2.3}\tag{E1.2.3}
\Delta(1)=1\otimes 1, \quad \Delta(a)=a\otimes 1+1\otimes a+\delta(a)
\end{equation}
for all $a\in L$. By embedding $L^1=L\oplus k1$ into $U(L)$
naturally, we can define $[f,g]$ for all elements in $f,g\in L^1$.
For any $k$-space $V$, let $V^{\otimes 2}$ denote $V\otimes V$.

\begin{lemma}
\label{xxlem1.3} Let $L$ be a Lie algebra, and let $\delta: L\to
L^{\otimes 2}$ be any linear map. Then $L$ is a coassociative Lie
algebra if and only if $(L^1,\Delta)$ is a counital coalgebra and
the equation
\begin{equation}
\label{E1.3.1}\tag{E1.3.1}
\Delta([a,b])=[\Delta(a), \Delta(b)]
\end{equation}
holds in $U(L)^{\otimes 2}$ for all $a,b\in L^1$.
\end{lemma}

\begin{proof} A direct calculation shows that $\Delta$ is coassociative
if and only if $\delta$ is coassociative. It is also easy to see that
\eqref{E1.1.1} is equivalent to \eqref{E1.3.1}.
\end{proof}

It follows from \eqref{E1.2.2} that the kernel
$\ker \delta$ is a Lie subalgebra of $L$. Let
$$\Phi(a,b):=
\delta([a,b])-(b_1\otimes [a, b_2]+[a,b_1]\otimes b_2+[a_1,b]\otimes
a_2+a_1\otimes [a_2,b]).$$ Then \eqref{E1.1.1} becomes
\begin{equation}
\label{E1.3.2}\tag{E1.3.2}
\Phi(a,b)=[\delta(a), \delta(b)].
\end{equation}

Let $L$ be a coassociative Lie algebra. If $\delta=0$, it is an
ordinary  Lie algebra. If $[\;, \;]= 0$, then $L$ is simply a
coassociative coalgebra without counit. Let $\LDA$ denote the
category of coassociative Lie algebras, $\LA$ denote the category of
Lie algebras and $\CA$ denote the category of coalgebras without
counit. Then both $\LA$ and $\CA$ are full subcategories of $\LDA$.
Now we give some simple examples of coassociative Lie algebras with
nontrivial Lie bracket $[\;,\;]$ and nontrivial coproduct $\delta$,
in which both sides of \eqref{E1.1.1} are trivially zero. More
complicated examples will be given in later sections.

\begin{example}
\label{xxex1.4} Let $L=kx\oplus ky$ with $[x,y]=y$ and
$\delta(x)=\lambda y\otimes y$ for any $\lambda\in k$ and
$\delta(y)=0$. Then $L$ is a coassociative Lie algebra. It is clear
that $(L,[\;,\;])$ is a Lie algebra and $(L,\delta)$ is a coalgebra.
Note that \eqref{E1.1.1} is trivial when $a=b$. So it suffices to
check \eqref{E1.1.1} for $a=x$ and $b=y$, in which case both sides
of \eqref{E1.1.1} are zero. It can be verified that, up to
isomorphism, this is the unique coassociative Lie algebra with an
underlying non-abelian Lie algebra of dimension 2.
\end{example}

Here is a slightly different situation.

\begin{lemma}
\label{xxlem1.5}
Let $(L, [\;,\;])$ be a Lie algebra and $(L,\delta)$ be a coalgebra.
Suppose that $\delta(L)\subset Z\otimes Z$ where $Z$ is the center of
$(L, [\;,\;])$ and that $\delta([L,L])=0$. Then $L$ is a
coassociative Lie algebra.
\end{lemma}

\begin{proof} The hypotheses imply that both sides of \eqref{E1.1.1}
are zero.
\end{proof}

\begin{example}
\label{xxex1.6} Suppose that $L$ is a Lie algebra containing a Lie
ideal $W$ such that both $W$ and $L/W$ are abelian. For example, $L$
is a Lie algebra of nilpotency 2. Let $\delta: L\xrightarrow{\pi}
L/W\to W^{\otimes 2}$ be a $k$-linear map where $\pi$ is the
quotient map. A direct calculation shows that $(\delta\otimes
1)\delta= (1\otimes \delta)\delta=0$. Thus $(L, \delta)$ is
coassociative. By the definition of $\delta$, the hypotheses of
Lemma \ref{xxlem1.5} hold; whence $L$ is a coassociative Lie
algebra. Below are two special cases.
\begin{enumerate}
\item
Let $L$ be the 3-dimensional Heisenberg Lie algebra with a basis
$\{x,y,z\}$ such that $[x,y]=z$ and $z$ is central. Define
$\delta(x)= z\otimes z$ and $\delta(y)=\lambda z\otimes z$, for some
$\lambda\in k$, and $\delta(z)=0$. By the above paragraph, $L$ is a
coassociative Lie algebra.
\item
Let $W_1$ be any Lie algebra and $W$ be a vector space. Let $\phi:
W_1 \to W$ be any $k$-linear map. Define a Lie bracket on $L:=W_1
\oplus W$ by
$$[w_1+w, t_1+t]=\phi([w_1,t_1]_{W_1})$$
for all
$w_1,t_1\in W_1$ and $w,t\in W$. Then $W$ is a Lie ideal of $L$ such
that both $L/W$ and $W$ are abelian. Then every $k$-linear map
$\delta: L\xrightarrow{\pi} W_1\to W\otimes W$ defines a
coassociative Lie algebra $L$.
\end{enumerate}
\end{example}

Next, we give an example of ``almost coassociative Lie algebra'',
which is dependent on the embedding of $L$ into an associative
algebra. If $A$ is any associative algebra, define $[\;,\;]_A$ to be
the commutator of $A$.

\begin{example}
\label{xxex1.7} Let $H$ be a bialgebra and let $H_{+}:=\ker \epsilon$
where $\epsilon$ is the counit of $H$. Let $L$ be a Lie subalgebra of
$(H_{+}, [\;,\;]_H)$ (for example, $L=H_{+}$). Define $\delta(a)=
\Delta(a)-1\otimes a-a\otimes 1$ for all $a\in H$. Then
$(H_{+}, \delta)$ is a coalgebra. Suppose that $\delta (L)\subset
L^{\otimes 2}$. Then $(L,\delta)$ is a coalgebra. However,
$(L, [\;,\;]_H, \delta)$ is generally not a coassociative Lie algebra.

Let $H$ be the 4-dimensional Taft Hopf algebra:
$$k\langle g,x\rangle/(xg+gx=0,\; g^2=1,\; x^2=0)$$
with $\Delta(g)=g\otimes g$, $\Delta(x)=x\otimes 1+g\otimes x$, and
$\epsilon(g)=1$, $\epsilon(x)=0$. Let $L$ be $H_{+}=kx\oplus
k(g-1)\oplus k (gx)$. Let $y=g-1$ and $z=gx$. Then $[x,y]_H=-2z$,
$[x,z]_H=0$, $[y,z]_H=2x$, $\delta(x)=y\otimes x$,
$\delta(y)=y\otimes y$, and $\delta(z)=z\otimes y$. Hence
$$[\delta(x), \delta(z)]_{H^{\otimes 2}}
=[y\otimes x, z\otimes y]_{H^{\otimes 2}}= yz\otimes xy-zy\otimes
yx=:(*).$$
Using the facts $[x,y]_H=-2z$ and $[y,z]_H=2x$, we have
$$(*)=yz\otimes xy-(yz-2x)\otimes
(xy+2z)=2yz\otimes z+2x\otimes xy+4 x\otimes z,$$ which is not in
$L\otimes L$ if we embed $L$ into $U(L)$. Hence \eqref{E1.1.1} does
not hold for $(L,[\;,\;]_H,\delta)$, and consequently, $L$ is not a
coassociative Lie algebra. After identifying $yz$ with $-z$, and
$xy$ with $-x-z$ in $H$, \eqref{E1.1.1} does hold in $H\otimes H$.
\end{example}

The enveloping algebra of a coassociative Lie algebra is defined
as follows.

\begin{definition}
\label{xxdef1.8}
Let $L$ be a coassociative Lie algebra. The enveloping algebra
of $L$, denoted by $U(L)$, is defined to be a bialgebra, whose
algebra structure equals that of the enveloping algebra of the
underlying Lie algebra $L$, namely,
$$U(L)=k\langle L\rangle /(ab-ba=[a,b], \forall \; a, b\in L),$$
and whose coalgebra structure is determined by
$$\Delta(a)= a\otimes 1+1\otimes a+\delta(a), \quad \epsilon(a)=0$$
for all $a\in L$. By \eqref{E1.3.1} it is easy to see that $U(L)$ is
a bialgebra. We will also use $U(L, \delta)$ to denote $U(L)$ if we
want to emphasize the coproduct $\delta$.
\end{definition}

It is clear that the assignment $L\to U(L)$ defines a functor from
$\LDA$ to ${\sf {BiAlg}}$ where ${\sf {BiAlg}}$ is the category of
bialgebras.

\begin{example}
\label{xxex1.9} If $\dim L=1$, then there are exactly two
coassociative Lie algebra structures on $L$, up to isomorphism. One
is determined by $\delta=0$. In this case, the enveloping algebra is
$U(L)=k[x]$ with $x$ being a primitive element. Consequently, $U(L)$
is a Hopf algebra. The other is determined by $\delta(x)=x\otimes
x$. Here, $U(L)=k[g]$ where $g=1+x$ and $g$ is a group-like
element in $U(L)$. In this case, $U(L)$ is not a Hopf algebra
because the group-like element $g$ is not invertible in $U(L)$.
\end{example}

Let $({\mathfrak g},\delta)$ be a coassociative Lie algebra with
underlying Lie algebra ${\mathfrak g}$. Then
Poincar{\'e}-Birkhoff-Witt (PBW) theorem holds for $U({\mathfrak
g},\delta)$, since, algebraically it is the usual enveloping algebra
$U({\mathfrak g})$. The difference between $U({\mathfrak g},
\delta)$ and $U({\mathfrak g})$ is their coalgebra structures. For
many examples of $({\mathfrak g},\delta)$, one can construct
explicitly a family of bialgebras $B(q)$ dependent on
$({\mathfrak g},\delta)$, where $q\in k$, such that
$B(1)=U({\mathfrak g},\delta)$ and $B(0)=U({\mathfrak g})$. Hence
$U({\mathfrak g},\delta)$ can be considered as a coalgebraic
deformation of $U({\mathfrak g})$. However, we will not pursue this
topic further.

Since $U(L)$ is generated by $L$ as an algebra, $U(L)$ is cocommutative
if and only if the underlying coalgebra $L$ is cocommutative. Similarly,
$U(L)$ is commutative if and only if the underlying Lie algebra $L$ is
abelian.

Let $\delta^n=(\delta\otimes 1^{\otimes n-1})
(\delta\otimes 1^{\otimes n-2})\cdots (\delta\otimes 1)\delta$.
Here is a list of definitions.

\begin{definition}
\label{xxdef1.10}
Let $L_1, L_2, L$ be coassociative Lie algebras.
\begin{enumerate}
\item
We say that $L_1$ and $L_2$ are {\it quasi-equivalent} if $U(L_1)$
is isomorphic to $U(L_2)$ as bialgebras.
\item
A Lie algebra ${\mathfrak g}$ is called {\it rigid} if every compatible
$\delta$-structure on ${\mathfrak g}$ is zero.
\item
A coalgebra $C$ is called {\it rigid} if every compatible
Lie-structure on $C$ is trivial.
\item
A coalgebra $(C,\delta)$ is called {\it anti-cocommutative} if
$\tau\delta=-\delta$, where the flip $\tau:C^{\otimes 2}\to
C^{\otimes 2}$ is defined by $\tau(a\otimes b)=b\otimes a$.
\item
The {\it nilpotency} of $L$, denoted by $n(L)$, is defined to be
the nilpotency of the underlying Lie algebra $L$.
\item
An element $x\in L$ is called {\it conilpotent} if $\delta^n(x)=0$
for some $n>0$. We say $L$ is {\it locally conilpotent} if every element
in $L$ is conilpotent.
\item
We call $L$ {\it $n$-conilpotent} if $\delta^n(L)=0$. The smallest
such $n$, denoted by $con(L)$, is called {\it conilpotency} of $L$.
\end{enumerate}
\end{definition}

\section{Results on enveloping algebras}
\label{xxsec2}

In this section we study some properties of the enveloping algebras
$U(L)$. Let $B$ be a bialgebra with coproduct $\Delta$. Define
$\delta_B: B\to B^{\otimes 2}$ by
$$\delta_B(x)=\Delta(x)-x\otimes 1-1\otimes x$$
for all $x\in B$.

\begin{definition}
\label{xxdef2.1}
A subspace $V$ in a bialgebra $B$ is called a
{\it $\delta$-space} of $B$ if
\begin{enumerate}
\item
$V$ is a Lie subalgebra of $(B, [\;,\;]_B)$,
\item
$\epsilon(V)=0$,
\item
$\delta_B(V)\subset V^{\otimes 2}$ inside $B^{\otimes 2}$, and
\item
$B$ is an ${\mathbb N}$-filtered algebra with an exhaustive filtration
defined by $F_n(B):=(k1+V)^n$, for $n\geq 0$, such that the associated
graded ring $gr_F B$ is isomorphic to the commutative polynomial ring
$k[V]$.
\end{enumerate}
\end{definition}

\begin{remark}
\label{xxrem2.2}
If $L$ is coassociative Lie algebra, then $L$ is $\delta$-space of $U(L)$.
In general, a $\delta$-space of $U(L)$ is not unique. See Corollary
\ref{xxcor2.6}.
\end{remark}

\begin{lemma}
\label{xxlem2.3} If $V$ is a $\delta$-space of $B$, then
$(V, [\;,\;]_B,\delta_B)$ is a
coassociative Lie algebra and $B\cong U(V)$ as bialgebras.
\end{lemma}

\begin{proof} Let $U$ be the usual enveloping algebra of the Lie
algebra $(V, [\;,\;]_B)$. Then there is an algebra homomorphism
$\phi: U\rightarrow B$ such that $\phi\mid_V=Id_V$. It follows from
Definition \ref{xxdef2.1}(c) that $B$ is generated by $V$ and that
the set $\{v_1^{n_1}\cdots v_d^{n_d}\mid n_i\geq 0\}$ is a
$k$-linear basis of $B$ where $\{v_1,\cdots,v_j,\cdots\}$ is a
$k$-linear basis of $V$. Since $\{v_1^{n_1}\cdots v_d^{n_d}\mid
n_i\geq 0\}$ is also a $k$-linear basis of $U$ by PBW theorem,
$\phi$ is an isomorphism of algebras. Note that $B$ is a bialgebra
and generated by $V$ as an algebra, one can defined a canonical
bialgebra structure $\Delta_U$ on $U$ via $\phi$ such that $\phi$ is
an isomorphism of bialgebras. Let $\delta_U(v)=\Delta_U(v)-v\otimes
1-1\otimes v$ for $v\in V$. Since $\delta_B(x)=\Delta_B(x)- x\otimes
1-1\otimes x$, $\delta_U(v)=\delta_B(v)$ for all $v\in V$ (we are
identifying the subspace $V\subset B$ with the subspace $V\subset U$
via the map $\phi\mid_V=Id_V$). By Definition \ref{xxdef2.1}(b), one
sees easily that $\delta_U(v)\in V^{\otimes 2}$ for all $v\in V$.
Since $U$ is a bialgebra (via the map $\phi$), \eqref{E1.3.1} holds.
Now by Lemma \ref{xxlem1.3} and Definition~\ref{xxdef1.8}, $(V,
[\;,\;]_U,\delta_U)$ is a coassociative Lie algebra with enveloping
algebra $U$. Since $(V, [\;,\;]_B,\delta_B)= (V,
[\;,\;]_U,\delta_U)$ by construction and $U\cong B$ as bialgebras,
the results follow.
\end{proof}

Let $(L,\delta)$ be a coalgebra (without counit). Let $L^1=k1\oplus
L$, and $\Delta: L^1\to (L^1)^{\otimes 2}$ be defined as
\eqref{E1.2.3}. Moreover, let $\epsilon: L^1\to k$ be defined by
$\epsilon(1)=1, \epsilon(x)=0$ for all $x\in L$. Then the assignment
$(L,\delta)\to (L^1,\Delta,\epsilon)$ defines a functor from $\CA$
to $\uCA$, where $\uCA$ is the category of counital coassociative
coalgebras. The following lemma is easy.

\begin{lemma}
\label{xxlem2.4}
Let $(L,\delta)$ be a coalgebra. Then $(L,\delta)$ is locally
conilpotent if and only if $(L^1,\Delta,\epsilon)$ is a connected
counital coalgebra.
\end{lemma}

\begin{proof}
Since $L$ is a sum of its finite dimensional subcoalgebras, we can
assume without loss of generality that $L$ is finite dimensional.
Note that $L$ can be identified with the quotient coalgebra $L^1/k1$.
By taking the $k$-linear dual, $L^*$ becomes a subalgebra (without a
unit) of $(L^1)^*$. In fact, $L^*$ is a maximal ideal of $(L^1)^*$ of
codimension $1$. Now the lemma is equivalent to the statement that
$L^*$ is a nilpotent ideal if and only if $L^*$ is the unique maximal
ideal of $(L^1)^*$, which is an easy ring-theoretical fact.
\end{proof}

Now we are ready to prove Theorem \ref{xxthm0.1}.

\begin{theorem}
\label{xxthm2.5} Let $L$ be a coassociative Lie algebra. Then the
following are equivalent:
\begin{enumerate}
\item
$L$ is locally conilpotent;
\item
$U(L)$ is a connected Hopf algebra; and
\item
$U(L)$ is a Hopf algebra.
\end{enumerate}
\end{theorem}

\begin{proof} (a) $\Rightarrow$ (b). Since $(L,\delta)$ is locally conilpotent,
$(L^1, \Delta,\epsilon)$ is a connected coalgebra by Lemma \ref{xxlem2.4}.
Since $U(L)$ is generated by $L^1$ as an algebra, $U(L)$ is connected as
a coalgebra. It follows from \cite[Lemma 5.2.10]{Mo} that a connected
bialgebra is automatically a Hopf algebra.

(b) $\Rightarrow$ (c). This is clear.

(c) $\Rightarrow$ (a). We proceed by contradiction. Suppose that
$U(L)$ is a Hopf algebra, but $(L,\delta)$ is not locally
conilpotent. Then $(L^1, \Delta)$ is not connected, whence its
coradical is strictly larger that $k1$. Pick a simple subcoalgebra
with counit of $L^1$, say $C$, which not equal to $k1$. Since $k$ is
algebraically closed, $C$ is isomorphic to a matrix coalgebra
$\bigoplus_{i,j=1,\ldots,n} kx_{ij}$ with
$$\Delta(x_{ij})=\sum_{s=1}^n x_{is}\otimes x_{sj}, \quad
{\text{and}}\quad \epsilon(x_{ij})=\delta_{ij}$$ for all $1\leq i,j
\leq n$. Here $\delta_{ij}$ is the Kronecker delta. Let
$y_{ij}=x_{ij}-\delta_{ij}$, for all $i,j$. Note that
$L=\ker(\epsilon: L^1\to k)$. Then $\bigoplus_{i,j=1,\ldots,n}k
y_{ij}\subset L$ is a simple subcoalgebra of $L$ such that
$\delta(y_{ij})=\sum_{s=1}^n y_{is}\otimes y_{sj}$ for all $1\leq
i,j \leq n$. Let $z_{ij}=S(x_{ij})$ for all $1\leq i,j\leq n$. Let
$X$ be the matrix $(x_{ij})_{n\times n}$ and $Z$ be the matrix
$(z_{ij})_{n\times n}$. Then the antipode axiom implies that
$XZ=ZX=I_{n}$ where $I_{n}$ is the identity $n\times n$-matrix. Note
that $L$ is a $\delta$-space of $U(L)$, and hence $U(L)$ has a
filtration defined by $F_n=(L^1)^n$ such that $\gr_F U(L)$ is
isomorphic to the commutative polynomial ring $k[L]$. One can extend
this filtration naturally from $U(L)$ to the matrix algebra
$M_n(U(L))$ such that $\gr_F (M_n(U(L))\cong M_n(k[L])$. Let $\gr$
also denote the leading terms of elements in $\gr_F (M_n(U(L))$.
Then the equation $XZ=I_{n}$ implies that $\gr(X)\gr(Z)=0$. Note
that $\gr(X)=(y_{ij})\in M_n(k[L])$ and thus $\det \gr(X)=\det
(y_{ij})$, which is nonzero in the commutative polynomial subring
$k[y_{ij}]\subset k[L]$. Consequently, the equation $\gr(X)\gr(Z)=0$
implies that $\gr(Z)=0$. Hence  $Z=0$, yielding a contradiction. The
assertion follows.
\end{proof}

\begin{corollary}
\label{xxcor2.6}
If ${\text{char}} \; k=0$, then every locally
conilpotent cocommutative coassociative Lie algebra
is quasi-equivalent to a Lie algebra.
\end{corollary}

\begin{proof} Let $L$ be any  locally conilpotent cocommutative
coassociative Lie algebra. Then $U(L)$ is a connected cocommutative
Hopf algebra by Theorem \ref{xxthm2.5}. By Milnor-Moore-Cartier-Kostant
Theorem \cite[Theorem 5.6.5]{Mo}, $U(L)$ is isomorphic to
$U({\mathfrak g})$ for some Lie algebra ${\mathfrak g}$. The assertion
follows.
\end{proof}

Recall that a Lie algebra ${\mathfrak g}$ is called {\it unimodular}
if $ad(x)$ has zero trace for all $x\in {\mathfrak g}$, where
$ad(x)\in \End_k(\mathfrak{g})$ is the $k$-linear map sending
$y\in\mathfrak{g}$ to $[x, y]$. Combining results of Koszul \cite{Ko} and
Yekutieli \cite[Theorem A]{Ye} (also see \cite[Proposition 6.3]{BZ}
and \cite[Theorem 5.3 and Lemma 4.1]{HVZ}), ${\mathfrak g}$ is
unimodular if and only if $U({\mathfrak g})[d]$ is the rigid
dualizing complex over $U({\mathfrak g})$ if $d:=\dim {\mathfrak g}$
is finite. By \cite[Proposition 6.3]{BZ} and \cite[Theorem
5.3]{HVZ}, ${\mathfrak g}$ is  unimodular if and only if the Hopf
algebra $U({\mathfrak g})$ is unimodular in the sense of \cite{LWZ},
if and only if $U({\mathfrak g})$ is Calabi-Yau, and if and only if
the homological integral of $U({\mathfrak g})$ (as defined in
\cite{LWZ}) is trivial. It is well-known that all Heisenberg Lie
algebras are unimodular, and that the 2-dimensional non-abelian Lie
algebra is not. Next we verify Theorem \ref{xxthm0.3}.

\begin{theorem}
\label{xxthm2.7}
Let ${\mathfrak g}$ be a finite dimensional unimodular Lie algebra.
Suppose $({\mathfrak g},\delta)$ is a coassociative Lie algebra
such that $\delta$ is conilpotent. Then the Hopf algebra $U({\mathfrak g},
\delta)$ is involutory.
\end{theorem}

\begin{proof} Let $H$ (respectively $K$) denote the Hopf algebra
$U({\mathfrak g},\delta)$ (respectively, $U({\mathfrak g})$). Let
$\mu_H$ and $\mu_K$ denote the Nakayama automorphisms of $H$ and $K$
respectively. Since the Nakayama automorphism is defined uniquely up
to an inner automorphism, and the units in $H$ are those elements in
$k^\times$, the Nakayama automorphism of $H$ (and of $K$) is unique.
Since $H=K$ as algebras by Definition \ref{xxdef1.8}, we have
$\mu_H=\mu_K$. Since ${\mathfrak g}$ is unimodular, $\mu_{K}=Id_K$
by \cite[Proposition 6.3(c)]{BZ}. As a consequence $\mu_H=Id_H$.

Since ${\mathfrak g}$ is unimodular, the homological integral of
$K$, denoted by $\int^l_{K}$, is trivial. Consequently, $\int^l_{K}$
equals the trivial module $K/({\mathfrak g})$ where $({\mathfrak
g})$ is the ideal of $K$ generated by subspace ${\mathfrak g}$. The
homological integral is only dependent on the algebra structure of
the Hopf algebra, so we have that $\int^l_H=H/({\mathfrak g})$. This
implies that $\int^l_H$ is trivial. Therefore the left winding
automorphism associated to $\int^l_H$, denoted by $\Xi^l_{\int^l}$,
is the identity map of $H$.

Combining the above with \cite[Theorem 0.3]{BZ}, we have that
$$Id_H=\mu_{H}=S_H^2\circ \Xi^l_{\int^l}=S_H^2\circ Id_H=S_H^2$$
where $S_H$ is the antipode of $H$. Hence $S_H^2=Id_H$, and $H$ is
involutory.
\end{proof}

Example \ref{xxex4.2} shows that $U({\mathfrak g},\delta)$ may not be
involutory if ${\mathfrak g}$ is not unimodular. This is another
way of showing that $U({\mathfrak g},\delta)$ is not isomorphic
to $U({\mathfrak g}')$ for any Lie algebra ${\mathfrak g}'$.
For the rest of this section we assume that ${\text{char}}\; k\neq 2$.

\begin{lemma}
\label{xxlem2.8}
Let $L$ be an anti-cocommutative coalgebra. Then
\begin{enumerate}
\item
$con(L)\leq 2$, and as a consequence, $L$ is conilpotent; and
\item
$\delta(L)\subset (\ker \delta)^{\otimes 2}$.
\end{enumerate}
\end{lemma}

\begin{proof} (a) A standard calculation by using Sweedler's notation
shows that
$$(1\otimes \delta)\tau\delta=(\tau\otimes 1)(1\otimes \tau)
(1\otimes \delta)\delta.$$
Since $\tau\delta=-\delta$ by assumption, the left-hand side of
the equation becomes $-(1\otimes \delta)\delta$ while the right-hand
side is $(1\otimes \delta)\delta$. As a consequence,
$(1\otimes \delta)\delta=0$.

(b) For any $x\in L$, write $\delta(x)=\sum_{i=1}^n x_i\otimes y_i$
for a minimal integer $n$. Then $\{x_i\}_{i=1}^n$ is linearly
independent. Since $con(L)\leq 2$,
$$0=(1\otimes \delta)\delta(x)=\sum_{i=1}^n x_i\otimes
\delta(y_i).$$
Since $\{x_i\}_{i=1}^n$ is linearly independent, $\delta(y_i)=0$
for all $i$. This means that $\delta(L)\subset L\otimes \ker \delta$.
Similarly, $\delta(L)\subset \ker \delta\otimes L$. The assertion follows.
\end{proof}

Lemma \ref{xxlem2.8} also implies a nice fact about the form of the
antipode.

\begin{proposition}
\label{xxprop2.9} Suppose that $L$ is anti-cocommutative and that
$U(L)$ is involutory. Then $S(x)=-x$ for all $x\in L$.
\end{proposition}

\begin{proof} It follows from Definition \ref{xxdef1.1} that the kernel
of $\delta$, denoted by $K$, is a Lie subalgebra of $L$. If $x\in K$,
then $\Delta(x)=x\otimes 1+1\otimes x$. The antipode axiom implies that
$S(x)=-x$. If $x\in L \setminus K$, it follows from Lemma
\ref{xxlem2.8}(b) that
$$\delta(x)=\sum_{i<j} a_{ij}(x_i\otimes x_j-x_j\otimes x_i),$$
for some $x_i\in K$ and some $a_{ij}\in k$. Hence
$$\Delta(x)=x\otimes 1+1\otimes x+\sum_{i<j}
a_{ij}(x_i\otimes x_j-x_j\otimes x_i).$$ Applying the antipode
axiom, and using the fact that $S(x_i)=-x_i$, we have that
$$0=S(x)+x+\sum_{i<j}a_{ij}(S(x_i)x_j-S(x_j)x_i)
=S(x)+x+\sum_{i<j}a_{ij}(-x_ix_j+x_jx_i).$$ So $S(x)=-x-Y$ where
$Y=\sum_{i<j}a_{ij}(-x_ix_j+x_jx_i)\in K$. Applying $S$ to
$S(x)=-x-Y$, and using the hypothesis that $S^2=Id$, we have
$x=-S(x)+Y$. Thus $Y=0$, and $S(x)=-x$.
\end{proof}

\section{Elementary properties of coassociative Lie algebras}
\label{xxsec3}

In this section some elementary properties of coassociative Lie
algebras are discussed. First we need some lemmas that will simplify
computations.

\begin{lemma}
\label{xxlem3.1} Let $L$ be a coassociative Lie algebra. Let
$\{x_i\}_{i\in I}$ be a totally ordered $k$-linear basis of $L$.
For any $a,b\in L$, write $\delta(a)=\sum_i x_i\otimes a_i$ and
$\delta(b)=\sum_ix_i \otimes b_i$. Then
\begin{enumerate}
\item
$[a_i,b_i]=0$ for all $i$; and
\item
$[a_i,b_j]+[a_j,b_i]=0$ for all $i,j$.
\end{enumerate}
If $\delta(L)\subset L'\otimes L''$ for some subspaces $L'$ and
$L''$ of $L$, then
\begin{enumerate}
\item[(c)]
$[\delta(a),\delta(b)]\in [L',L']\otimes [L'',L'']$;
\item[(d)]
$\Phi(a,b)\in [L',L']\otimes [L'',L'']$; and
\item[(e)]
$\delta([L,L])\subset [L',L]\otimes L''+L'\otimes [L'',L]
+[L',L']\otimes [L'',L'']$.
\end{enumerate}
\end{lemma}

\begin{proof} We compute $[\delta(a),\delta(b)]$ in $U(L)^{\otimes 2}$
as follows
$$\begin{aligned}
\; [\delta(a),\delta(b)]
&=\sum_{i,j} x_i x_j\otimes a_i b_j-\sum_{i,j}x_ix_j\otimes b_i a_j\\
&=\sum_{i<j} x_i x_j\otimes a_i b_j+\sum_{j>i} x_j x_i\otimes a_j b_i
+\sum_{i} x_i^2\otimes a_i b_i\\
&\quad -\sum_{i<j} x_i x_j\otimes b_i a_j-\sum_{j>i} x_j x_i\otimes b_j a_i
-\sum_{i} x_i^2\otimes b_i a_i\\
&=\sum_{i<j}x_ix_j\otimes (a_ib_j-b_ia_j)+\sum_i x_i^2\otimes [a_i,b_i]\\
&\quad +\sum_{j>i} (x_i x_j+[x_j,x_i])\otimes a_j b_i -
\sum_{j>i} (x_ix_j+[x_j, x_i])\otimes b_j a_i \\
&=\sum_{i<j} x_ix_j\otimes ([a_i,b_j]+[a_j,b_i])+
\sum_i x_i^2\otimes [a_i,b_i]\\
&\quad \qquad\qquad +\sum_{i<j}[x_i,x_j]\otimes (b_ja_i-a_jb_i).
\end{aligned}$$
Since $\{x_ix_j\}_{i\leq j}$ are linearly independent in $U(L)/L$
and $[x_i,x_j]\in L$ for all $i,j$, we have
$[a_i,b_j]+[a_j,b_i]=0$, $[a_i,b_i]=0$. Parts (a) and (b) follow.

Now assume that $\delta(L)\subset L'\otimes L''$. By the above
computation and parts (a,b),
$[\delta(a),\delta(b)]=\sum_{i<j}[x_i,x_j]\otimes (b_ja_i-a_jb_i)$,
which is in $[L',L']\otimes U(L)$, as we can assume $x_i\in L'$
whenever $a_i$ (or $b_i$) is nonzero. By symmetry,
$[\delta(a),\delta(b)]\in U(L)\otimes [L'',L'']$. Hence
$[\delta(a),\delta(b)]\in [L',L']\otimes [L'',L'']$. This is part
(c). Part (d) follows from the equation
$\Phi(a,b)=[\delta(a),\delta(b)]$.

Part (e) follows from part (c) and \eqref{E1.1.1}.
\end{proof}

If $V$ and $W$
are subspaces of a vector space $A$, let $V/W$ denote $V/(V\cap W)$.

\begin{definition}
\label{xxdef3.2}
\begin{enumerate}
\item
A subspace $V$ of a Lie algebra $L$ is said to have {\it small
centralizer} if $(\ker ad(x))\cap V$ has dimension 1 for all $x\in
V\setminus \{0\}$.
\item
Let $Z$ be a Lie ideal of $L$. A subspace $V\subset L$ is said to
have {\it small centralizer modulo $Z$} if the quotient
space $V/Z$ in $L/Z$ has small centralizer.
\end{enumerate}
\end{definition}

\begin{proposition}
\label{xxprop3.3}
Let $L$ be a coassociative Lie algebra and $Z$ be
a Lie ideal of $L$. Suppose that $L'$ and $L''$ are two subspaces of
$L$ such that
\begin{enumerate}
\item
$[L', Z]=[L'',Z]=0$,
\item
$L'$ and $L''$ have small centralizers modulo $Z$, and
\item
$\delta(L) \subset L''\otimes L' +(Z\otimes L+L\otimes Z)$.
\end{enumerate}
 Then
 $\dim \delta(L)/(Z\otimes L+L\otimes Z)\leq 1.$
\end{proposition}

\begin{remark}
\label{xxrem3.4}
Since $Z$ is a Lie ideal of $L$, we have that $L/Z$
is a quotient Lie algebra, but $L/Z$ may not be a quotient of the
coassociative Lie algebra $L$.
\end{remark}

\begin{proof}[Proof of Proposition \ref{xxprop3.3}]
Let $W=(Z\otimes L+L\otimes Z)\cap \delta(L)$. Without loss of
generality we may assume that $\delta(L)\neq W$. Let $a$ and $b$ be
any two elements in $L$.  The assertion is equivalent to

{\it Claim: $\delta(a)$ and $\delta(b)$ are linearly dependent in
$\delta(L)/W$.}

\noindent {\it Proof of the Claim}: If $\delta(a)$ or $\delta(b)$ is
in $W$, the claim is obvious, so we assume during the proof that
$\delta(a)$ and $\delta(b)$ are not in $W$.

In the rest of the proof, we will pick a $k$-linear basis
$\{z_j\}_{j\geq 1}$ of $Z$, extend it to a basis $\{x_i\}_{i\geq
1}\cup \{z_j\}_{j\geq 1}$ of $L''+Z$ where $x_i\in L''\setminus Z$,
then extend it to a basis $\{x_i\}_{i\geq 1}\cup \{z_j\}_{j\geq
1}\cup \{l_s\}_{s \geq 1}$ of the whole space $L$ where $l_s\in
L\setminus (L''+Z)$. For simplicity, we use integers to index the
basis elements. Since $\delta(L) \subset L''\otimes L' +(Z\otimes
L+L\otimes Z)$, for any $a\in L$,
$$\delta(a)=\sum_{i\geq 1} x_i\otimes a_i+\sum_{j\geq 1} z_j\otimes
a'_j+\sum_{s\geq 1} l_s\otimes a''_s$$ where $a_i\in L'+Z$, $a'_j\in
L$ and $a''_s\in Z$.

Case 1: Suppose that $\delta(a), \delta(b) \in u\otimes L'+(Z\otimes
L+L\otimes Z)$ for some $u\in L''\setminus Z$. In this case we can
choose $x_1=u$. By the choice of the $k$-linear basis, we can write
$$\delta(a)=u\otimes a_1+\sum_{i\geq 2} x_i\otimes a_i
+\sum_{j\geq 1} z_j \otimes a'_j+\sum_{s\geq 1} l_s\otimes a''_s,$$
where $a_1\in L'+Z$, $a_i\in Z$ for all $i\geq 2$, $a'_j\in L$ and
$a''_s\in Z$ for all $j,s$. Similarly,
$$\delta(b)=u\otimes b_1+\sum_{i\geq 2} x_i\otimes b_i
+\sum_{j\geq 1} z_j \otimes b'_j+\sum_{k\geq 1} l_s\otimes b''_s$$
where $b_1\in L'+Z$, $b_i\in Z$ for all $i\geq 2$, and $b'_j\in L$
and $b''_s\in Z$ for all $j,s$. By Lemma \ref{xxlem3.1}(a), $[a_1,
b_1]=0$. Since $L'$ has small centralizer modulo $Z$, $b_1\in
ka_1+Z$. Thus $\delta(a)$ and $\delta(b)$ are linearly dependent in
$\delta(L)/W$.

Case 2: Suppose that $\delta(a), \delta(b) \in L''\otimes
u+(Z\otimes L+L\otimes Z)$ for some $u\in L'\setminus Z$. Case 2 is
equivalent to Case 1 by symmetry. Hence the claim follows by Case 1.

Case 3: Suppose that $\delta(a)\in u\otimes L'+(Z\otimes L+L\otimes
Z)$ for some $u\in L''\setminus Z$. In this case, we can choose
$x_1=u$, and $\delta(a)$ can be written as in Case 1. In particular,
$\delta(a)\in L''\otimes a_1+(Z\otimes L+L\otimes Z)$. Write
$$\delta(b)=u\otimes b_1+\sum_{i\geq 2} x_i\otimes b_i
+\sum_{j\geq 1} z_j \otimes b'_j+\sum_{s\geq 1} l_s\otimes b''_s$$
where $b_i\in L'+Z$ for all $i\geq 1$, and $b'_j\in L$, $b''_s\in Z$
for all $j,s$. By Lemma \ref{xxlem3.1}(a), $[a_1, b_1]=0$. Since
$L'$ has small centralizer modulo $Z$, $b_1=\lambda a_1+z$ for some
$\lambda\in k$ and $z\in Z$. Replacing $b_1$ by $b_1- \lambda a_1$,
we may assume that $b_1\in Z$. By Lemma~\ref{xxlem3.1}(b), for every
$i\geq 2$, $[a_1, b_i]=-[a_i, b_1]\in Z$ since $b_1$ is in $Z$.
Since $L''$ has small centralizer modulo $Z$, $b_i=\lambda_i
a_1+z_i$, for $\lambda_i\in k$ and $z_i\in Z$ for all $i\geq 2$.
Thus
$$\delta(b)=u\otimes b_1+\sum_{i\geq 2} x_i\otimes (\lambda_i a_1+z_i)
+\sum_{j\geq 1} z_j \otimes b'_j+\sum_{s\geq 1} l_s\otimes b''_s$$
which is in $L''\otimes a_1+(Z\otimes L+L\otimes Z)$. Therefore both
$\delta(a)$ and $\delta(b)$ are in $L''\otimes a_1+(Z\otimes L+
L\otimes Z)$. The claim follows from Case 2.

Case 4: Suppose that either $\delta(a)$ or $\delta(b)$ is in
$L''\otimes u+(Z\otimes L+L\otimes Z)$ for some $u\in L'\setminus Z$.
The claim follows by symmetry and Case 3.

Case 5 [the general case]: By the choice of $k$-linear basis, we can
write
$$\begin{aligned}
\delta(a)&=\sum_{i\geq 1} x_i\otimes a_i
+\sum_{j\geq 1} z_j \otimes a'_j+\sum_{s\geq 1} l_s\otimes a''_s,\\
\delta(b)&=\sum_{i\geq 1} x_i\otimes b_i +\sum_{j\geq 1} z_j \otimes
b'_j+\sum_{s\geq 1} l_s\otimes b''_s,
\end{aligned}
$$
where $a_i, b_i\in L'+Z$, $a'_j, b'_j\in L$ and $a''_s, b''_s\in Z$
for all $i,j,s$. Without loss of generality, we may assume that
$a_1\in L'\setminus Z$. By Lemma \ref{xxlem3.1}(a), $[a_1, b_1]=0$.
Since $L'$ has small centralizer modulo $Z$, $b_1=\lambda a_1+z$ for
some $\lambda\in k$ and where $z\in Z$. Replacing $b_1$ by $b_1-
\lambda a_1$, we may assume that $b_1\in Z$. By Lemma
\ref{xxlem3.1}(b), for every $i\geq 2$, $[a_1, b_i]=-[a_i, b_1]\in
Z$ since $b_1$ are in the Lie ideal $Z$. Since $L'$ has small
centralizer modulo $Z$, $b_i=\lambda_i a_1+z_i$, for $\lambda_i\in
k$ and $z_i\in Z$ for all $i\geq 2$. Together with the fact $b_1\in
Z$, we have that $\delta(b)\in L''\otimes a_1+(Z\otimes L+L\otimes
Z)$. The claim now follows from Case 4.
\end{proof}

For a subset $S\subset L$, the centralizer of $S$ in $L$ is defined to be
$$C_S(L)=\{y\in L\mid [x,y]=0, \forall\; x\in S\}.$$

\begin{lemma}
\label{xxlem3.5} If there is an element $a\in L$ such that
$\delta(a)= x\otimes y\neq 0$, then $\delta(L)\subset C_{\{x\}}(L)\otimes
C_{\{y\}}(L)$.
\end{lemma}

\begin{proof} By symmetry, it suffices to show that $\delta(L)\subset
L\otimes  C_{\{y\}}(L)$. Pick a basis $\{x_i\}$ of $L$ such that $x_1=x$.
Then $\delta(a)=x_1\otimes y$. For any $b\in L$, write
$\delta(b)=\sum_i x_i\otimes b_i$. By Lemma \ref{xxlem3.1}(a),
$[y,b_1]=0$. For any $i\geq 2$, by Lemma \ref{xxlem3.1}(b), $[y,b_i]
=-[0,b_1]=0$. The assertion follows.
\end{proof}

\begin{proposition}
\label{xxprop3.6}
 Let $C$ be the
coradical of a coassociative Lie algebra $L$. Then $[C,C]=0$.
As a consequence, cosemisimple coalgebras are rigid.
\end{proposition}

\begin{proof} Since $k$ is algebraically closed, $C=\bigoplus_i M_{n_i}(k)$
for a set of positive integers $\{n_i\}_{i\in I}$. Let $x,y\in C$;
we need to show $[x,y]=0$. By linearity, we may assume that $x$ and
$y$ are some basis elements in $C$. We need to consider two cases.

Case 1: We have that $x$ and $y$ are in the same matrix
subcoalgebra, say $M_{n}(k)$. If $n=1$, $x=y$. The assertion is
trivial. Now assume that $n>1$. Then we may assume that $x=x_{ij}$
and $y=x_{kl}$ for some $i,j,k,l$. Consider $\delta(x_{1j})=\sum_s
x_{1s} \otimes x_{sj}$ and $\delta(x_{2l})=\sum_t x_{2t} \otimes
x_{tl}$. By Lemma \ref{xxlem3.1}(b), $[x_{sj}, x_{tl}]=0$ for all
$s,t$. The assertion follows.

Case 2: We have that $x$ and $y$ are in different matrix
subcoalgebras. Then we may assume that  $x=x_{ij}\in M_{n_1}(k)$ and
$y=y_{kl}\in M_{n_2}(k)$. Consider $\delta(x_{1j})=\sum x_{1s}
\otimes x_{sj}$ and $\delta(y_{1l})=\sum y_{1t} \otimes y_{tl}$. By
Lemma \ref{xxlem3.1}(b), $[x_{sj}, y_{tl}]=0$ for all $s,t$. The
assertion follows.
\end{proof}

The following lemma is also true. The proof is omitted since it is
straightforward and somewhat similar to the proof of Proposition
\ref{xxprop3.6}.

\begin{lemma}
\label{xxlem3.7} If $C_1$ and $C_2$ are subcoalgebras of a
coassociative Lie algebra such that $C_1\cap C_2=\{0\}$, then
$[\delta(C_1), \delta(C_2)]=0$.
\end{lemma}

\begin{lemma}
\label{xxlem3.8} Let $L$ be a coalgebra. Then we have the following
statements.
\begin{enumerate}
\item
For every $x\in L$, write $\delta(x)=\sum_{i=1}^n w_i\otimes v_i$ for
a minimal $n$. Then  $\sum_i k\delta(v_i)\subset L\otimes V_x$ and
$\sum_i k\delta(w_i)\subset W_x\otimes L$ for some subspaces
$V_x\subset \sum_i kv_i$ and $W_x\subset\sum_i kw_i$ of dimension no
more than $\dim\delta( L)$.
\item
If $\delta(L)$ is 1-dimensional and $L$ is not 2-conilpotent, then
$\delta(L)$ has a basis element of the form $T\otimes T$ for
some $0\neq T\in L$.
\end{enumerate}
\end{lemma}

\begin{proof} (a)
Let $\{y_t\}_{t=1}^m$ be a basis of $\sum_s k\delta(v_s)$ for
some $m\leq \dim \delta(L)$. Then there are elements $a_1, \ldots, a_m
\in L$ such that
$$\sum_i \delta(w_i)\otimes v_i= \sum_i w_i\otimes \delta(v_i)
=\sum_{t=1}^m a_t\otimes y_t\in (\sum_{t=1}^m ka_t)\otimes L\otimes L.$$
This implies that $\delta(w_i)\in (\sum_t ka_t)\otimes L$ for each $i$.
Since $\{y_t\}_{t=1}^m$ is a basis of $\sum_s k\delta(v_s)$, the equation
$$\sum_i w_i\otimes \delta(v_i)
=\sum_{t=1}^m a_t\otimes y_t$$
implies that $\sum_t ka_t\subset \sum_i kw_i$. Therefore the second
assertion follows by taking $W_x=\sum_t k a_t$.
The first assertion is similar.

(b) Pick $0\neq \Omega\in \delta(L)$, and let $\{x_i\}$ be a finite
set of linearly independent elements in $L$ such that $\Omega
=\sum_{i,j} a_{ij} x_i\otimes x_j$. Since $\delta(L)$ is
1-dimensional, $\delta(x_i)=b_i \Omega$ for some $b_i\in k$. Pick
$x\in L$ such that $\delta(x)=\Omega$ and $(\delta\otimes
1)\delta(x)\neq 0$ (since $L$ is not $2$-conilpotent). Then
$$(\delta\otimes 1)\delta(x)=\sum_{i, j} a_{i, j} b_i
\Omega\otimes x_j =\Omega\otimes (\sum_{i, j}a_{ij}b_ix_j)=\Omega\otimes T$$
where $T:=\sum_{i,j}a_{ij}b_ix_j\in L$, and
$$(1\otimes \delta)\delta(x)=\sum_{i, j}a_{ij} x_i\otimes b_j \Omega=
(\sum_{i, j}a_{ij}b_jx_i)\otimes \Omega=S\otimes \Omega$$ where
$S:=\sum_{i, j}a_{ij}b_jx_i$. By coassociativity, $\Omega\otimes
T=S\otimes \Omega$. This implies that Hence $\Omega=c' T\otimes T$
and $S=c''T$  for some $c',c''\in k^\times$. Since $k$ is
algebraically closed, we can choose $c'=1$ by a scalar change of
$T$.
\end{proof}

Now we prove the main result of this section. Let
$$\left\{e:=\begin{pmatrix} 0&1\\0&0\end{pmatrix}, \quad f:=\begin{pmatrix}
0&0\\1&0\end{pmatrix},\quad  h:=\begin{pmatrix}
1&0\\0&-1\end{pmatrix}\right\}$$ be a standard $k$-basis of $sl_2$.

\begin{theorem}
\label{xxthm3.9} The simple Lie algebra $sl_2$ is rigid.
\end{theorem}

\begin{proof}  By
using the standard  basis of $sl_2$, it is straightforward to check
that $sl_2$ has small centralizers (details are omitted). Let
$L=(sl_2,\delta)$ be a coassociative Lie algebra. We need to show
that $\delta=0$. By Proposition \ref{xxprop3.3} for $L'=L''=L=sl_2$
and $Z=0$, $\dim \delta(L)\leq 1$. To avoid the triviality, we
assume that $\dim \delta(L)=1$ and let $\Omega\in \delta(L)$ be a
nonzero element.

If $L$ is not 2-conilpotent, then Lemma \ref{xxlem3.8} says that
$\Omega=T\otimes T$ where $T=t_1e+t_2f+t_3h$ for some
$t_1,t_2,t_3\in k$. Since $\Omega\neq 0$, not all $t_i$ are zero.
Suppose $\delta(e)=a\Omega$, $\delta(f) =b\Omega$ and
$\delta(h)=c\Omega$ for some $a,b,c\in k$. By \eqref{E1.1.1}, we
have
$$\begin{aligned}
2b\Omega&= \delta (2f)=\delta([f,h])\\
&=[f\otimes 1+1\otimes f, c\Omega]+[b\Omega, h\otimes 1+1\otimes h]\\
&=\{-2bt_1 e+(2ct_3+2bt_2)f-ct_1 h\}\otimes T\\
&\qquad\qquad\qquad
+T\otimes \{-2bt_1 e+(2ct_3+2bt_2)f-ct_1 h\}.
\end{aligned}
$$
Hence $b(t_1,t_2,t_3)=(-2bt_1, 2bt_2+2ct_3,-ct_1)$, or
$$
M\begin{pmatrix} t_1\\t_2\\t_3\end{pmatrix} =0
\quad {\text{where}}\quad M=\begin{pmatrix}b& 0&0\\ 0&b& 2c\\
c&0& b\end{pmatrix} .$$ Since not all $t_1,t_2,t_3$ are zero, the
determinant of the matrix $M$, which is $b^3$, is zero. Hence $b=0$.
Since $e$ and $f$ plays a very similar role, by symmetry, $a=0$. By
\eqref{E1.1.1} and the fact $\delta(e)=\delta(f)=0$, we have that
$\delta(h)=\delta([e,f])=0$. Thus $c=0$, whence $\delta=0$, yielding
a contradiction. Therefore $\delta$ is 2-conilpotent.

Since $\delta$ is 2-conilpotent, by Theorem \ref{xxthm2.5},
$U(sl_2,\delta)$ is a connected Hopf algebra of GK-dimension 3. By
comparing with the list in the classification of connected Hopf
algebras of GK-dimension 3 \cite[Theorem 1.3]{Zh2}, $U(sl_2,\delta)$
must be isomorphic to $U(sl_2)$. Since $U(sl_2)$ is cocommutative,
$(sl_2,\delta)$ must be cocommutative. Hence
$$\begin{aligned}
\Omega&=a_{11}e\otimes e+a_{12}(e\otimes f+f\otimes e)+
a_{13}(e\otimes h+h\otimes e)\\&\qquad\qquad +a_{22}f\otimes
f+a_{23}(f\otimes h +h\otimes f)+ a_{33}h\otimes h\neq 0.
\end{aligned}
$$
Since $\delta(L)$ is 1-dimensional, the kernel $L_0=\ker(\delta)$ is
a 2-dimensional Lie subalgebra of $sl_2$. By an elementary
computation, it is easy to verify that any 2-dimensional Lie
subalgebra of $sl_2$ is either (i) $kf+kh$, or (ii) $ke+kh$, or
(iii) $k(e+ah)+k(-4af+h)$ for some $a\in k^\times$. In the first
case, $(\delta\otimes 1)\Omega=a_{11}\delta(e)\otimes
e+a_{12}\delta(e)\otimes f+a_{13}\delta(e)\otimes h$ and without
loss of generality, we can assume that $\delta(e)=\Omega$. Since
$(sl_2,\delta)$ is 2-conilpotent, $(\delta\otimes 1) (\Omega)=0$,
which implies that $a_{11}=a_{12}=a_{13}=0$. It is easy to see that
$$\begin{aligned}
-\Phi(e,h)&=-\delta(-2e)+[\delta(e), h\otimes 1+1\otimes h]\\
&= 2\Omega+4a_{22}f\otimes f+2a_{23}(f\otimes h+h\otimes f).
\end{aligned}$$
The equations
$\Phi(e,h)=[\delta(e),\delta(h)]=[\delta(e),0]=0$ imply that
$a_{22}=a_{23}=a_{33}=0$. Hence $\Omega=0$, a contradiction. The
assertion follows. The second case is similar.

The final case is when $e+ah , h-4af\in \ker(\delta)$ for some $a\in
k^\times$. Let
$$\begin{aligned}
e'&=\begin{pmatrix} 1 & 0\\ -2a& 1 \end{pmatrix}
\begin{pmatrix} 0& 1\\ 0& 0 \end{pmatrix}
\begin{pmatrix} 1 & 0\\ -2a& 1 \end{pmatrix}^{-1}=
\begin{pmatrix} 2a & 1\\ -4a^2& -2a \end{pmatrix}\\
h'&=\begin{pmatrix} 1 & 0\\ -2a& 1 \end{pmatrix}
\begin{pmatrix} 1& 0\\ 0& -1 \end{pmatrix}
\begin{pmatrix} 1 & 0\\ -2a& 1 \end{pmatrix}^{-1}=
\begin{pmatrix} 1& 0\\ -4a& -1 \end{pmatrix}\\
f'&=\begin{pmatrix} 1 & 0\\ -2a& 1 \end{pmatrix}
\begin{pmatrix} 0& 0\\ 1& 0 \end{pmatrix}
\begin{pmatrix} 1 & 0\\ -2a& 1 \end{pmatrix}^{-1}=
\begin{pmatrix} 0& 0\\ 1& 0 \end{pmatrix}=f.
\end{aligned}$$
Then $\{e',f', h'\}$ is a new standard basis of $sl_2$. It is clear
that $e'=e+2a h-4a^2 f\in \ker(\delta)$ and $h'=h-4af\in
\ker(\delta)$. Thus it is equivalent to the second case. Combining
all these cases, the assertion follows.
\end{proof}

Similar to Theorem \ref{xxthm3.9} we show the following.

\begin{theorem}
\label{xxthm3.10}
Write $gl_2=sl_2\oplus kz$ where $kz$ is the center of $gl_2$.
If $(gl_2,\delta)$ is a coassociative Lie algebra, then $\delta\mid_{sl_2}=0$
and $\delta(z)=a z\otimes z$ for some scalar $a\in k$.
\end{theorem}

\begin{proof}[Sketch of Proof] Some tedious computations are omitted
in the following proof.

First of all, the $\delta$ given in the theorem gives rise to
a coassociative Lie algebra structure on $gl_2$. Now we assume that
$(gl_2,\delta)$ is a coassociative Lie algebra.

Applying Proposition \ref{xxprop3.3} to $(L,L', L'', Z)=(gl_2,
sl_2,sl_2,kz)$, we obtain that
$$\dim (\delta(gl_2)/(z\otimes gl_2+gl_2\otimes z))\leq 1.$$
Therefore there is an element $\Omega\in sl_2 \otimes sl_2$
such that, for every $x\in gl_2$,
\begin{equation}
\label{E3.10.1}\tag{E3.10.1}
\delta(x)=\sigma(x)\otimes z+z\otimes \tau(x)+\lambda(x) \Omega
\end{equation}
for some $\sigma(x)\in sl_2, \tau(x)\in gl_2$ and $\lambda(x)\in k$.
Both $\sigma(x)$ and $\tau(x)$ are uniquely determined by
\eqref{E3.10.1}. If $\Omega\neq 0$, then $\lambda(x)$ is
also uniquely determined by \eqref{E3.10.1}. Setting $x=z$ in
\eqref{E3.10.1}, we have
\begin{equation}
\label{E3.10.2}\tag{E3.10.2}
\delta(z)=\sigma(z)\otimes z+z\otimes \tau(z)+\lambda(z) \Omega.
\end{equation}
Equation \eqref{E1.2.2} for $(a,b)=(z,x)$ implies that
$$\begin{aligned}
0=&[\sigma(z),x]\otimes z+z\otimes [\tau(z),x]+
[\lambda(z)\Omega,x\otimes 1+1\otimes x]\\
&+[\sigma(z),\sigma(x)]\otimes z^2+z^2
\otimes [\tau(z),\tau(x)]\\
&+[\lambda(z)\Omega, \sigma(x)\otimes z]+
[\lambda(z)\Omega, z\otimes \tau(x)]\\
&+[\sigma(z)\otimes z, \lambda(x)\Omega]+
[z\otimes \tau(z), \lambda(x)\Omega].
\end{aligned}
$$
Since the terms in the above equation live in different $k$-subspaces
of $U(gl_2)\otimes U(gl_2)$, we have $[\sigma(z),x]=0=[\tau(z),x]$
for all $x\in sl_2$. Thus $\sigma(z)=0$ and $\tau(z)\in kz$.
In this case, \eqref{E1.2.2} becomes, for every $x\in gl_2$,
\begin{equation}
\label{E3.10.3}\tag{E3.10.3}
0=\lambda(z)([\Omega,x\otimes 1+1\otimes x]+[\Omega, \sigma(x)\otimes z]
+[\Omega, z\otimes \tau(x)]).
\end{equation}

First we claim that $\lambda(z)\Omega= 0$. If not, we may assume that
$\Omega\neq 0$ and $\lambda(z)=1$. Using the fact that terms live in
different $k$-subspaces of $U(gl_2)\otimes U(gl_2)$,
\eqref{E3.10.3} implies that
$$[\Omega,x\otimes 1+1\otimes x]=[\Omega, \sigma(x)\otimes z]=
[\Omega, z\otimes \tau(x)]=0$$
for all $x\in gl_2$. A computation shows that the first equation
implies that $\Omega=c(h\otimes h+2(e\otimes f+f\otimes e))$ for some
$0\neq c\in k$. The second and third equations
imply that $\sigma(x)=0$ and $\tau(x) \in kz$. Going back to
\eqref{E3.10.1}, we have, for every $x\in gl_2$,
$$\delta(x)=\phi(x)z\otimes z+\lambda(x) \Omega$$
for some $\phi(x)\in k$. For any $x,y\in sl_2$, \eqref{E1.2.2} says that
$$\delta([x,y])=[x\otimes 1+1\otimes x, \lambda(y)\Omega]+
[\lambda(x)\Omega, y\otimes 1+1\otimes y]=0.$$
Since $sl_2=[sl_2,sl_2]$, $\delta\mid_{sl_2}=0$.
The coassociativity on $z$ shows that $\Omega \otimes z=0$,
a contradiction. Therefore we proved our claim.

For the rest of the proof, we have $\lambda(z)\Omega=0$ and $\delta(z)=
\sigma(z)\otimes z+z\otimes \tau(z)$.
Then $kz$ is an ideal of the coassociative Lie algebra $(gl_2, \delta)$
and $(gl_2/kz, \overline{\delta})\cong (sl_2,\overline{\delta})$ is a
quotient coassociative Lie algebra where $\overline{\delta}$ is the
induced coproduct. By Theorem \ref{xxthm3.9}, $\overline{\delta}=0$.
This means that $\Omega=0$. For each $x\in gl_2$, write
$$\delta(x)=\sigma_1(x)\otimes z+z\otimes \sigma_2(x)+\nu(x) z\otimes z$$
for some $\sigma_1(x),\sigma(x)\in sl_2$ and some $\nu(x)\in k$.
By \eqref{E1.2.2}, we have
$$\begin{aligned}
\delta([x,y])=&[\sigma_1(x),y]\otimes z+z\otimes [\sigma_2(x),y]
+[x,\sigma_1(y)]\otimes z+z\otimes [x,\sigma_2(y)]\\
&+[\sigma_1(x),\sigma_1(y)]\otimes z^2+z^2\otimes [\sigma_2(x),\sigma_2(y)]
\end{aligned}$$
for all $x,y\in gl_2$. Setting $y=z$, we have $[x,\sigma_i(z)]=0$
for all $x\in sl_2$. This implies that $\sigma_i(z)=0$ for $i=1,2$.
Setting $x,y\in sl_2$, we have that
$$[\sigma_1(x),\sigma_1(y)]=0=[\sigma_2(x),\sigma_2(y)]$$
and that $\delta(sl_2)\subset sl_2\otimes z+z\otimes sl_2$.
Since $sl_2$ has small centralizers, $\dim \sigma_1(sl_2)\leq 1$
and $\dim \sigma_2(sl_2)\leq 1$. Combining above facts, there exist
$w_1,w_2\in sl_2\setminus \{0\}$, $c\in k$ and linear maps
$\phi_1,\phi_2:sl_2\to k$, such that
$$\begin{aligned}
\delta(z)&=c z\otimes z, \\
\delta(x)&= \phi_1(x) w_1\otimes z+\phi_2(x) z\otimes w_2
\end{aligned}
$$
for all $x\in sl_2$. Let $\{e,f,h\}$ be the standard basis of
$sl_2$, and write $w_1=a e+bf+ch\neq 0$ for some $a,b,c\in k$. A
calculation using explicit Lie product of the elements $e,f$, and
$h$ shows that \eqref{E1.2.2} implies that $\phi_1=0$. By symmetry,
$\phi_2=0$. Thus the assertion follows.
\end{proof}

\section{Examples}
\label{xxsec4}

We present several families of coassociative Lie algebras in this
section. One-dimensional ones are listed in
Example \ref{xxex1.9}. Here is the 2-dimensional case.

\begin{example}
\label{xxex4.1} If $\dim L=2$, then there are two Lie algebra
structures on $L$, up to isomorphism. Namely, $L$ is either abelian
or non-abelian.

If $L$ is abelian, then the classification of coassociative Lie
algebra structures on $L$ is equivalent to the classification of
coalgebra structures on $L$. It is easy to show that
$\delta$-structure in $L$ is isomorphic to one of the following:
\begin{enumerate}
\item[(4.1.1)]
$\delta=0$;
\item[(4.1.2)]
$(L,\delta)$ is cosemisimple;
\item[(4.1.3)]
$L=kx_1\oplus kx_2$ and $\delta(x_1)=x_1\otimes x_1$,
$\delta(x_2)=0$;
\item[(4.1.4)]
$L=kx_1\oplus kx_2$ and $\delta(x_1)=x_1\otimes x_1$,
$\delta(x_2)=x_1\otimes x_2+ x_2\otimes x_1$;
\item[(4.1.5)]
$L=kx_1\oplus kx_2$ and $\delta(x_1)=0$, $\delta(x_2)=x_1\otimes x_1$.
\end{enumerate}

If $L$ is non-abelian, then $L$ has a basis $\{x_1,x_2\}$ such that
$[x_1,x_2]=x_2$. We have two cases.

Case 1: $\delta(x_2)=0$. We are only interested in nonzero
$\delta$-structures. Write $\delta(x_1)=\sum_{i,j}a_{ij}x_i
\otimes x_j\neq 0$. In this case
$$\Phi(x_1,x_2)= -[\delta(x_1), x_2\otimes 1+1\otimes x_2]
=-(a_{11}(x_2\otimes x_1+x_1\otimes x_2)+(a_{12}+a_{21}) x_2\otimes
x_2).$$ By \eqref{E1.3.2},
$\Phi(x_1,x_2)=[\delta(x_1),\delta(x_2)]=0$. Hence $a_{11}=0$ and
$a_{12}+a_{21}=0$. Let $a=a_{12}$ and $b=a_{22}$. We have that
$\delta(x_1)=a x_1\otimes x_2 -ax_2\otimes x_1+b x_2\otimes x_2$.
Now coassociativity of $\delta$ shows that $a=0$. Thus
$\delta(x_2)=0$ and $\delta(x_1)=b x_2\otimes x_2$. Up to a base
change, we may assume $b=1$. Hence this is Example \ref{xxex1.4}.

Case 2: $\delta(x_2)\neq 0$. Since $L$ is 2-dimensional and
non-nilpotent, it has small centralizers. By Lemma \ref{xxprop3.3},
$\delta(L)$ is 1-dimensional. Since $\delta(x_2) \neq 0$, by
replacing $x_1$ by $x_1-a x_2$ for some suitable $a\in k$, we have
$\delta(x_1)=0$. Write $\delta(x_2)=\sum_{i,j}b_{ij}x_i\otimes
x_j\neq 0$. In this case,
$$\begin{aligned}
0=[\delta(x_1),\delta(x_2)]
&=\Phi(x_1,x_2)=\delta(x_2)-[x_1\otimes 1+1\otimes x_1, \delta(x_2)]\\
&=b_{11}x_1\otimes x_1+
b_{12} x_1\otimes x_2+b_{21}x_2\otimes x_1+b_{22}x_2\otimes x_2\\
&\quad -(b_{12}x_1\otimes x_2+b_{21}x_2\otimes x_1+2b_{22} x_2\otimes x_2)\\
&=b_{11}x_1\otimes x_1-b_{22}x_2\otimes x_2.
\end{aligned}$$
Hence $b_{11}=b_{22}=0$ and $\delta(x_2)=b_{12} x_1\otimes
x_2+b_{21}x_2\otimes x_1$. The coassociativity of $\delta$
implies that $b_{12}=b_{21}=0$. Therefore $\delta=0$ in this case, yielding
a contradiction.

Combining these two cases, the only nonzero $\delta$-structure  on
the 2-dimensional non-abelian Lie algebra is the one in Example
\ref{xxex1.4}, up to isomorphisms.

One nice fact in 2-dimensional case is that $\delta$ is always
cocommutative. If $\delta$ is conilpotent, then $L$ is quasi-equivalent
to a Lie algebra by Corollary \ref{xxcor2.6}.
\end{example}

Next, we consider some higher-dimensional examples.

\begin{example}
\label{xxex4.2} Let ${\mathfrak g}$ be a 3-dimensional Lie algebra
with a $k$-linear basis $\{x,y,z\}$ such that its Lie structure is
determined by
$$[x,y]=y, \quad [z,y]=0, \quad  [z,x]=-z+\lambda y,$$
for any $\lambda\in k$. Let $L=({\mathfrak g},\delta)$ where the
coproduct $\delta$ is determined by
$$\delta(x)=\delta(y)=0, \quad \delta(z)=x\otimes y-y\otimes x.$$
It is routine to check that $L$ is a coassociative Lie algebra
(using Definition \ref{xxdef1.1}). It is obvious that $\delta$ is
conilpotent and anti-cocommutative. Let $H$ be the enveloping
algebra $U(L)$. It follows from the antipode axiom that $S(x)=-x$,
$S(y)=-y$, and $S(z)=-z+y$. Hence $S^2(z)=z-2y$ and $H$ is not
involutory. By Theorem \ref{xxthm2.7}, ${\mathfrak g}$ is not
unimodular, which can also be verified directly.
\end{example}

Let ${\mathfrak h}_{2n+1}$ be the $(2n+1)$-dimensional Heisenberg
Lie algebra with a standard basis $\{x_1,\cdots,x_n, y_1,\cdots,y_n,
z\}$. Here $[x_i,y_i]=z$ for all $i$, and all other brackets are
zero. Let $A=(a_{ij})$,
$B=(b_{ij})$, $C=(c_{ij})$ and $D=(d_{ij})$ denote $n\times
n$-matrices over $k$, and let $E=(e_i)$ be an $n$-column vector over $k$.

\begin{example}
\label{xxex4.3} Each of the following $\delta$ defines a
coassociative coalgebra structure on ${\mathfrak h}_{2n+1}$ such
that $({\mathfrak h}_{2n+1},\delta)$ is a coassociative Lie algebra.
\begin{enumerate}
\item
For every $i$, $\delta(x_i)=\delta(z)=0$, and
$$\delta(y_i)=
\sum_j(a_{ij}x_j+b_{ij}y_j)\otimes z+\sum_j z\otimes
(c_{ij}x_j-b_{ij}y_j) +e_i z\otimes z,$$
where the coefficient matrices $A=(a_{ij})$, $B=(b_{ij})$, $C=(c_{ij})$
and $E=(e_i)$ satisfy
 \begin{enumerate}
 \item[(i)]
 $BA=B^2=BC=0$;
 \item[(ii)]
 $BE=0$;
 \item[(iii)]
 $A+A^{\tau}+C+C^{\tau}=0$;
 \item[(iv)]
 $AB^{\tau}=BA^{\tau}$; and
 \item[(v)]
 $CB^{\tau}=BC^{\tau}$.
 \end{enumerate}
\item
For every $i$,  $\delta(x_i)=e_i z\otimes z$, $\delta(z)=0$, and
$\delta(y_i)= \sum b_{ij} (x_j\otimes z+z\otimes x_j)$, where the
coefficient matrix $B=(b_{ij})$ satisfies $B=B^{\tau}$.
\item
For every $i$, $\delta(x_i)=0$, $\delta(z)=z\otimes z$, and
$$\delta(y_i)=\sum_{j} (a_{ij}x+b_{ij} y_j)\otimes z+\sum_{j}
z\otimes (-a_{ji}x_j+(\delta_{ij}-b_{ij})y_j),$$ where the
coefficient matrices $A=(a_{ij})$ and $B=(b_{ij})$ satisfy the
conditions $B^2=B$; $BA=A$; and  $BA^{\tau}=0$.
\end{enumerate}
\end{example}

\begin{proof} (a) Easy computations show that
$$\begin{aligned}
(\delta\otimes 1) &\delta(y_i)\\
&=\sum_j b_{ij}\delta(y_j)\otimes z\\
&=\sum_{j,s} (b_{ij}a_{js}x_s+b_{ij}b_{js}y_s)\otimes z\otimes z
+\sum_{j,s} z\otimes (b_{ij}c_{js}x_s-b_{ij}b_{js}y_s)\otimes z\\
&\qquad\qquad
+\sum_j b_{ij}e_j z\otimes z\otimes z, \;\quad  {\text{and}}\\
(1\otimes \delta) &\delta(y_i)\\
&=\sum_{j} z\otimes (-b_{ij})\delta(y_j)\\
&=\sum_{j,s}z\otimes (-b_{ij}a_{js}x_s-b_{ij}b_{js}y_s)\otimes
z+\sum_{j,s} z\otimes z\otimes (-b_{ij}c_{js}x_s+b_{ij}b_{js}y_s)\\
&\qquad\qquad  -\sum_j b_{ij}e_j z\otimes z\otimes z.
\end{aligned}
$$
Coassociativity is equivalent to equations $BA=B^2=BC=0$ and $BE=0$.
To check the condition \eqref{E1.1.1} we note that \eqref{E1.1.1} is
trivial when $(a,b)=(z,z), (x_i, z), (y_i,z)$, and $(x_i,x_j)$. It
suffices to verify \eqref{E1.1.1} for $(a,b)=(x_i,y_j)$ and
$(a,b)=(y_i,y_j)$ for all $1\leq i,j\leq n$.

If $(a,b)=(x_i,y_j)$, we have that
$$\begin{aligned}
{\text{LHS of \eqref{E1.1.1}}}&=
\delta([x_i,y_j])=\delta(\delta_{ij}z)=0, \; {\text{and}}\\
{\text{RHS of \eqref{E1.1.1}}}&=
[x_i\otimes 1+1\otimes x_i, \delta(y_j)]
= b_{ji} z\otimes z- b_{ji} z\otimes z=0.
\end{aligned}
$$
Hence \eqref{E1.1.1} holds for $(a,b)=(x_i,y_j)$.

If $(a,b)=(y_i,y_j)$, we have that
$$\begin{aligned}
{\text{LHS of \eqref{E1.1.1}}}&=\delta([y_i,y_j])=\delta(0)=0,
\; {\text{and}}\\
{\text{RHS of \eqref{E1.1.1}}}&=[y_i\otimes 1,
\delta(y_j)]+[1\otimes y_i,
\delta(y_j)]\\
&\quad + [\delta(y_i),y_j\otimes 1]+[\delta(y_i),1\otimes y_j]+
[\delta(y_i),\delta(y_j)]\\
&=a_{ji}z\otimes z+c_{ji}z\otimes z\\
&\quad +a_{ij}z\otimes z+c_{ij}z\otimes z\\
&\quad +(\sum_s a_{is}b_{js}-b_{is}a_{js})z\otimes z^2+
(\sum_s -c_{is}b_{js}+b_{is}c_{js})z^2\otimes z.
\end{aligned}
$$
Hence \eqref{E1.1.1} holds if and only if $A+A^{\tau}+C+C^{\tau}=0$,
$AB^{\tau}=BA^{\tau}$ and $CB^{\tau}=BC^{\tau}$. This completes the
proof of (a).

The proofs of (b) and (c) are similar and therefore omitted.
\end{proof}

We consider one last example. Let $U_n$ be the strictly upper
triangular $n\times n$-matrix coalgebra, namely, it is the coalgebra
with basis $\{x_{ij}\}_{1\leq i<j\leq n}$ such that
$$\delta(x_{ij})=\sum_{i<s<j}x_{is}\otimes x_{sj}
\qquad {\text{for all $1\leq i,j\leq n$}}.$$ In the following
proposition, let $E=(e_i), F=(f_i)$ and $G=(g_i)$ be three arbitrary
vectors in $k^{n-1}$. For $1\leq i<j\leq n$, define
$$a_{ij}=g_i+g_{i+1}+\cdots + g_{j-1}.$$
It follows from the definition that $a_{is}+a_{sj}=a_{ij}$
for all $1\leq i<s<j\leq n$.

\begin{example}
\label{xxex4.4} Let $n\geq 3$. Then the following anti-commutative
$k$-bilinear map $[\;,\;]: U_{n}^{\otimes 2}\to U_{n}$ defines a Lie
algebra structure on $U_n$ such that $(U_n,[\;,\;])$ is a
coassociative Lie algebra.
\begin{align}
\label{E4.4.1}\tag{E4.4.1}
\; [x_{1n},x_{1n}]&=0,\\
\label{E4.4.2}\tag{E4.4.2}
[x_{st},x_{ij}]&=0 \qquad\qquad\qquad
{\text{ if $(s,t)\neq (1,n)$ and $(i,j)\neq (1,n)$}},\\
\label{E4.4.3}\tag{E4.4.3}
[x_{1n}, x_{ij}]&= a_{ij} x_{ij} \qquad \qquad
{\text{ if $(i,j)\neq (1,n), (1,n-1), (2,n)$}},\\
\label{E4.4.4}\tag{E4.4.4}
[x_{1n},x_{1n-1}]&= a_{1n-1} x_{1n-1}+\sum_{i=1}^{n-1} e_i x_{ii+1},\\
\label{E4.4.5}\tag{E4.4.5}
[x_{1n},x_{2n}]&= a_{2n} x_{2n}+\sum_{i=1}^{n-1} f_i x_{ii+1}.
\end{align}
\end{example}

\begin{proof} First we prove that $(U_n,[\;,\;])$ is a Lie algebra.
Let $K=\bigoplus_{(i,j)\neq (1,n)} kx_{ij}\subset L$. By definition,
we have  $[K,K]=0$ and $[L,K]=[K,L]\subset K$. Since we define
$[\;,\;]$ to be anti-commutative, it suffices to show the Jacobi
identity
$$[a, [b,c]]=[b,[a,c]]+[[a,b],c]$$
for all $a, b,c\in U_n$. If $a,b,c\in K$, then the Jacobi identity
is trivially true. If $a=b=x_{1n}$ and $c\in K$, the Jacobi identity
is also true since it is true for all $a=b$. The Jacobi identity is
stable under permutation and thus the remaining case to consider is
when $a=x_{1n}$ and $b,c\in K$. In this case,
$$\begin{aligned}
\; [a,[b,c]]&=[a,0]=0,\\
[b,[a,c]]+[[a,b],c]&\in [K,K]+[K,K]=\{0\}.
\end{aligned}
$$
Hence the Jacobi identity holds and $(U_n,[\;,\;])$ is a
Lie algebra.

To prove $(U_n,[\;,\;])$ is a coassociative Lie algebra, we need
to verify \eqref{E1.1.1}. By linearity, it suffices to
check \eqref{E1.1.1} for cases listed in \eqref{E4.4.1}-\eqref{E4.4.5}.

Case 1: If $(a,b)=(x_{1n},x_{1n})$, \eqref{E1.1.1} is automatic
(for any $a=b$).

Case 2: Suppose that $(a,b)=(x_{st},x_{ij})$ for $(s,t)\neq (1,n)$
and $(i,j)\neq (1,n)$. Since $\delta(L)\subset K\otimes K$ and
$[K,K]=0$, both sides of \eqref{E1.1.1} are zero.

Case 3: Suppose that $(a,b)=(x_{1n},x_{ij})$ for $(i,j)\neq (1,n)$.
Using the fact $[K,K]=0$, we have
$$\begin{aligned}
{\text{LHS of \eqref{E1.1.1}}}&=\delta([x_{1n},x_{ij}])=a_{ij}\
\sum_{i<s<j} x_{is}\otimes x_{sj},\\
{\text{RHS of \eqref{E1.1.1}}}&=[x_{1n}\otimes 1+1\otimes x_{1n},
\sum_{i<s<j} x_{is}\otimes x_{sj}]\\
&=(a_{is}+a_{sj}) \sum_{i<s<j} x_{is}\otimes x_{sj}
=a_{ij}\sum_{i<s<j} x_{is}\otimes x_{sj}.
\end{aligned}
$$
Hence \eqref{E1.1.1} holds. This takes care of cases in
\eqref{E4.4.3}-\eqref{E4.4.5}.

Combining all the above cases we have checked \eqref{E1.1.1}.
Therefore $(U_n,[\;,\;])$ is a coassociative Lie algebra.
\end{proof}

\subsection*{Acknowledgments}
The authors thank Chelsea Walton for reading an earlier version of
this paper and for her useful comments and thank Milen Yakimov
for suggesting them the name ``coassociative Lie algebra'' for the
main object studied in this paper.
The authors also thank the referee for his/her valuable comments
and suggestions. Part of the research was
completed when J.J. Zhang visited Fudan University in the Fall of
2009, in the Spring of 2010, and in the Summer of 2011. D.-G. Wang
was supported by the National Natural Science Foundation of China
(No. 10671016  and  11171183) and the Shandong Provincial Natural
Science Foundation of China (No. ZR2011AM013). J.J. Zhang and G.
Zhuang were supported by the US National Science Foundation (NSF
grant No. DMS 0855743).

\providecommand{\bysame}{\leavevmode\hbox to3em{\hrulefill}\thinspace}
\providecommand{\MR}{\relax\ifhmode\unskip\space\fi MR }
\providecommand{\MRhref}[2]{%

\href{http://www.ams.org/mathscinet-getitem?mr=#1}{#2} }
\providecommand{\href}[2]{#2}

\end{document}